\documentclass{article}
\usepackage[utf8]{inputenc}
\usepackage{amsmath}
\usepackage{amsthm}
\usepackage{amssymb}
\usepackage{amsfonts}
\usepackage{enumerate}
\usepackage{hyperref}
\usepackage{tabularx}
\usepackage{graphicx}
\usepackage{multicol}
\usepackage{mathrsfs}
\usepackage{booktabs}
\usepackage{makecell}
\usepackage{array}
\usepackage{bigstrut}
\usepackage{caption}
\usepackage{subcaption}
\usepackage{caption}
\usepackage[inline]{asymptote}
\usepackage{biblatex}
\usepackage[final]{microtype}
\newtheorem{theorem}{Theorem}[section]

\newtheorem{corollary}[theorem]{Corollary} 

\theoremstyle{definition}
\newtheorem{definition}{Definition}[section]

\theoremstyle{definition}
\newtheorem{example}[definition]{Example}

\theoremstyle{remark}
\newtheorem{remark}[definition]{Remark}

\title{Existence of Circle Packings on Translation Surfaces}
\author{Anton Levonian}
\date{December 2023}

\begin{document}

\maketitle

\begin{abstract}
A translation surface is a surface formed by identifying edges of a collection of polygons in the complex plane that are parallel and of equal length using only translations. We determined that the same circle packing can be realized on varying translation surfaces in a certain stratum. We also determined possible complexities of contacts graphs and provide a bound on this complexity in some low-genus strata. Finally, we established the possibility of certain contacts graphs’ complexities in strata with genus greater than $2$.\end{abstract}

\section{Introduction}
Translation surfaces are very common and important types of Riemann surfaces, and can be defined through gluing together a collection of polygons in the complex plane without applying any rotations or reflections \cite{https://doi.org/10.48550/arxiv.1411.1827}. For example, a torus is a translation surface, formed by identifying opposite sides of a square. Translation surfaces have applications in billiards on rational polygons (\cite{https://doi.org/10.48550/arxiv.math/0609392}) and geodesic flow (\cite{Massart_2022}). The introduction of Veech surfaces in 1989 (see \cite{Veech1989}) has increased research on translation surfaces, and they have since been extensively studied (see \cite{https://doi.org/10.48550/arxiv.math/0106251, https://doi.org/10.48550/arxiv.1809.10769,https://doi.org/10.48550/arxiv.1712.09066,https://doi.org/10.48550/arxiv.2005.00099,https://doi.org/10.48550/arxiv.1411.1827,https://doi.org/10.48550/arxiv.math/0609392}).

Before we elaborate any further, we will describe a few topological and geometric properties that will be freely discussed throughout the paper. We define the \emph{genus} of a surface as the maximal number of cuts along simple closed curves that can be made without making the surface disconnected. Translation surfaces of genus greater than 2 also contain \emph{singular points}, whose geometry is different from that of other points on the surface. 

The results presented in this paper mainly concern circle packings on different translation surfaces with similar properties. We give more background on these objects in Section \ref{background}. We then settle questions regarding both the existence of packings on distinct surfaces in Section \ref{packingexistence} and the simplest packings that can be realized on certain families of translation surfaces in Section \ref{simplecontactsgraphs}. Finally, we outline directions that future research may take in Section \ref{researchdirections}.

\section{Background}
\label{background}
\subsection{Definition and Basic Properties of Translation Surfaces}
We begin by defining translation surfaces, the geometry of which will be the main subject of the paper.
\begin{definition}
    A \emph{translation surface} is formed by identifying opposite sides of polygons from $\mathcal{P}$, where $\mathcal{P}$ is a collection of several polygons in the complex plane such that $\mathcal{P}$ has an even number of sides and opposite sides of $\mathcal{P}$ are parallel and of equal length.
\end{definition}
A particular family of translation surfaces, the square-tiled surfaces, will be of importance.
\begin{definition}
    A \emph{square-tiled surface} is a translation surface for which $\mathcal{P}$ is formed by joining together opposite sides of congruent squares.
\end{definition}
Multiple polygons may define the same translation surface:
\begin{remark}
\label{equivalenceoftranslationsurfacesremark}
     Let $\mathcal{P}$ be a collection of polygons in the complex plane. Cut polygons in $\mathcal{P}$ along straight lines and translate the resulting polygons so that newly-created edges are identified to create a new collection of polygons $\mathcal{P}_1.$ After identification of opposite edges, $\mathcal{P}$ and $\mathcal{P}_1$ define the same surface (see Wright \cite{https://doi.org/10.48550/arxiv.1411.1827}).
\end{remark}
We now introduce singular points: 
\begin{definition}
    A \emph{singular point} is a point with at which the angle is greater than $2\pi$ to which multiple vertices of $\mathcal{P}$ get mapped. 
\end{definition}
\begin{definition}
    The \emph{cone angle} of a singular point is the total angle measure at this point. It is the sum of the angles of all the vertices of $\mathcal{P}$ which get mapped to the singular point.
\end{definition}

It is a well-known result, outlined by Zorich in \cite{https://doi.org/10.48550/arxiv.math/0609392}, that the total angle at any point on a translation surface will always be a multiple of $2\pi.$

\begin{definition}
    The \emph{order}, represented by $\delta,$ of a singular point $v$ with cone angle $2k\pi$ is $\delta(v)=k-1.$
\end{definition}

\begin{example}
Consider a square in the complex plane, and identify its opposite sides. This yields a torus, a surface of genus $1.$ All 4 vertices of the square will be mapped to a single point, with angle $4 \cdot \frac{\pi}{2}=2 \pi.$ Since this point's angle is not greater than $2\pi$, it is not a singular point.
\end{example}

\subsection{Genus and Stratum of Translation Surfaces}
We begin by introducing a well-known theorem that relates the genus of a translation surface to the order of its singular points
(see Massart \cite{Massart_2022}).
\begin{theorem} (Gauss-Bonnet Theorem).
\label{Gausbonnet}
Let $X$ be a translation surface with $k$ singular points $v_i$ each of order $\delta(v_i)$, and let $\chi(X)$ be the Euler characteristic of $X.$ Then the following is true: $$\sum_{i=1}^k \delta(v_i)+\chi(X)=0.$$ 
\end{theorem}

Theorem \ref{Gausbonnet} will be useful for calculating the genus, $g$, of $X$ by recalling the well-known formula $$\chi(X)=2-2g.$$

\begin{example}
\label{boundary-lessRTS}
    Consider the square-tiled surface with $3$ squares shown in Figure \ref{fig2}. All of the vertices of each of the 3 squares will be mapped to a single singular point $v$ of angle $12\cdot \frac{\pi}{2}=6\pi$. So, $\delta(v)=2,$ and since $\chi(X)=-\sum \delta(v_i),$ the Euler characteristic in this case is $-2$. We now see that the genus is $g=2.$
\end{example}

\begin{figure}[htp]
\centering
\includegraphics[width=.25\textwidth]{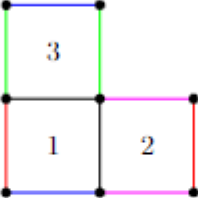}
    \caption{Genus $2$ 3-squared surface.}
    \label{fig2}
\end{figure}

One more important topic related to the genus of a translation surface is the surface's stratum.
\begin{definition}
    A \emph{stratum}, denoted by $\mathcal{H}(\kappa)$ where $\kappa$ is a partition of $\chi(X)$, is the set of all translation surfaces with singular point orders equal to the parts of $\kappa.$
\end{definition}

\begin{example}
    The surface in Example \ref{boundary-lessRTS} has 1 singular point of order $2$ and no other singular points of nonzero order, so it is a member of $\mathcal{H}(2).$ Meanwhile, the surface in Figure \ref{fig4} is a member of $\mathcal{H}(1,1),$ having 2 singular points each of order 1.

\begin{figure}[htp]
    \centering
     \includegraphics[width=.33\textwidth]{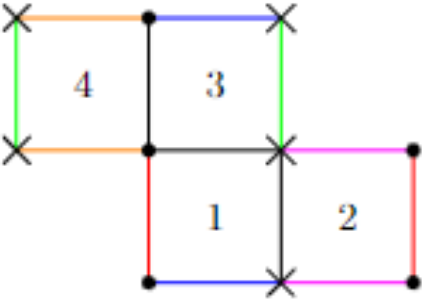}
       \caption{Translation surface in $\mathcal{H}(1,1).$}
       \label{fig4}
    
\end{figure}
Because $(1,1)$ and $(2)$ are the only partitions of $2,$ all genus $2$ translation surfaces are either in $\mathcal{H}(1,1)$ or in $\mathcal{H}(2).$
\end{example} 

\begin{example}
    All of the regular octagon's vertices are mapped to a single point, so its angle is $\pi(8-2)=6\pi.$ Hence, this point is a singular point of order $2$ and the regular octagon is in $\mathcal{H}(2).$
\end{example}
We now make the following important observation, which will be useful in Section \ref{simplecontactsgraphs}, about polygons which are mapped to $\mathcal{H}(2)$ surfaces.
\begin{remark}
\label{ststooctaremark}
    It follows from Proposition 1.16 of Wright \cite{https://doi.org/10.48550/arxiv.1411.1827} and Remark \ref{equivalenceoftranslationsurfacesremark} that every $\mathcal{H}(2)$ surface results from gluing opposite edges of a rotationally symmetric octagon in the plane. This is because any polygon forming a cylinder and a slit torus upon identification of opposite edges can be cut and reglued into an octagon with rotational symmetry, and a cutting and regluing in the other direction is also possible for every octagon. This is illustrated in Figure \ref{fig:ststooctagon} below:
    \begin{figure}[htp]
        \centering
         \begin{subfigure}[b]{0.4\textwidth}
         \centering
       \includegraphics{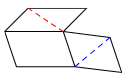}
           \end{subfigure}
           \hspace{.1\textwidth}
           \begin{subfigure}[b]{.4\textwidth}
           \centering
          \includegraphics{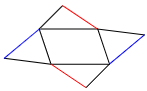}
            \end{subfigure}
        \caption{The 2 polygons above are identified to the same $\mathcal{H}(2)$ surface.}
        \label{fig:ststooctagon}
    \end{figure}

\end{remark}
\subsection{Circle Packings}
Of particular interest are circle packings, popularized by Thurston (see \cite{thurston2014three}), on translation surfaces.
\begin{definition}
\label{circlepackingdef}
    A \emph{circle packing} is a collection of circles with disjoint interiors with a \emph{contacts graph} $G$ so that each vertex of $G$ corresponds to a circle and 2 circles are tangent if and only if their corresponding vertices on $G$ are connected by an edge. 
\end{definition}

In Definition \ref{circlepackingdef}, note that the radii and placement of the circles is not determined by $G$.

\begin{example}
    Figure \ref{fig5} contains a possible circle-packing on the 3-squared surface from Figure \ref{fig2}, along with the corresponding contacts graph in Figure \ref{fig6}. We will denote this particular packing as $C_3$ throughout.
    \begin{figure}[htp]
    \centering
\includegraphics[width=.25\textwidth]{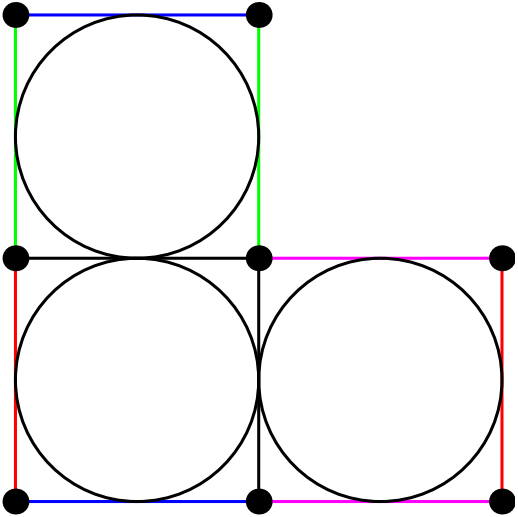}
       \caption{A circle packing $C_3$ on a surface in the complex plane.}
       \label{fig5}
       \end{figure}
\end{example}

    \begin{figure}[htp]
     \centering
     \includegraphics[width=0.3\textwidth]{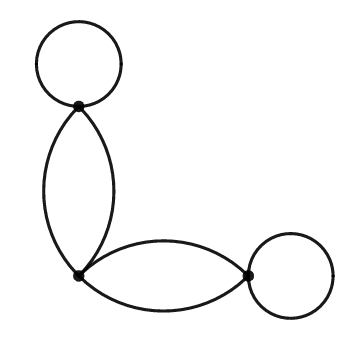}
        \caption{Contacts graph for $C_3.$}
        \label{fig6}
\end{figure}

\begin{remark}
    The metric on a translation surface is not a flat smooth metric near singular points of nonzero order (see Wright \cite{https://doi.org/10.48550/arxiv.1411.1827}). Hence, circles containing singular points do not correspond with circles in the complex plane.
\end{remark}

\begin{example}
    Figure \ref{fig7} depicts a collection of circular sectors on a 3-squared $\mathcal{P}$ in the complex plane which, when opposite sides of $\mathcal{P}$ are identified, yield a circle with radius less than $\frac{1}{2}$ centered at a singular point.

    \begin{figure}[htp]
    \centering
    \includegraphics[width=.25\textwidth]{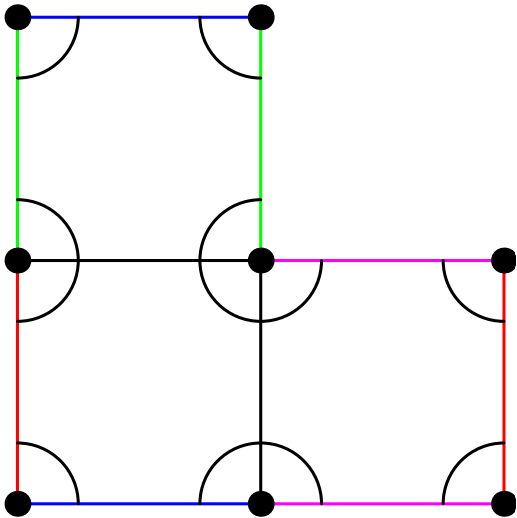}
       \caption{Visualizing a circle centered at a singular point.}
       \label{fig7}
       \end{figure}

\end{example}
We now introduce a concept that will be useful in defining circle packing equivalence.
\begin{definition}
    A \emph{tangency pattern} is the pattern of intersection of segments connecting tangency points between pairs of circles. Consider 2 circle packings $C_1$ and $C_2$ with the same contacts graph $G_C$ and 2 quadruplets of not necessarily distinct circles $(t_1,u_1,v_1,w_1)$ and $(t_2,u_2,v_2,w_2)$ in $C_1$ and $C_2$, respectively so that:
    \begin{enumerate}
        \item $(t_i,u_i)$ and $(v_i,w_i)$ are pairs of tangent circles for $i\le 2$.
        \item $t_1$ and $t_2$ correspond to the same vertex of $G_C,$ and likewise for $u_i,v_i,$ and $w_i$.
    \end{enumerate} 
    Let $t_1u_1$ represent the segment connecting tangency points of $t_1$ and $u_1,$ and likewise define $v_1w_1, t_2u_2,$ and $v_2,w_2.$
    The tangency patterns of $C_1$ and $C_2$ are the same if and only if, for all $(t_i,u_i,v_i,w_i)$ in $C_i,$ $t_1u_1$ intersects $v_1w_1$ if and only if $t_2u_2$ intersects $v_2w_2$. 
\end{definition}
\begin{example}
    In Figure \ref{fig8}, 2 circle packings on 2 different translation surfaces of the same stratum have the same contact graph and tangency patterns:
    \begin{figure}[htp]
     \centering
     \begin{subfigure}[b]{0.25\textwidth}
         \centering
         \includegraphics[width=\textwidth]{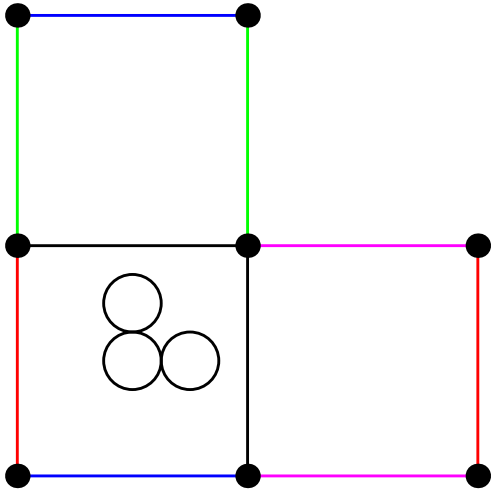}
         \caption{3-squared surface.}
         \label{fig8a}
     \end{subfigure}
     \hspace{.2\textwidth}
     \begin{subfigure}[b]{0.33\textwidth}
         \centering
         \includegraphics[width=\textwidth]{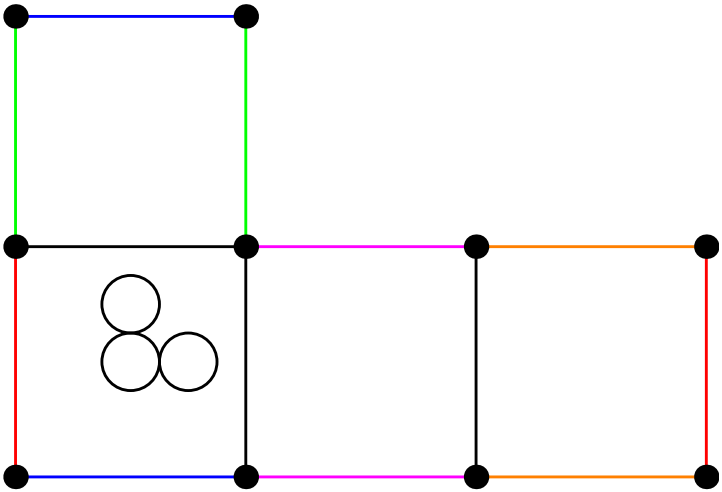}
         \caption{4-squared surface.}
         \label{fig8b}
     \end{subfigure}
     \caption{Equivalence of trivial circle packings.}
     \label{fig8}
     \end{figure}

\end{example}

The following theorem (see Stephenson \cite{stephenson2005introduction}) is instrumental in studying circle packings, and provides motivation for our explorations in Section \ref{packingexistence}.
\begin{theorem} (Koebe-Andreev-Thurston Theorem)
    Given a Riemann sphere with a graph $G,$ there exists a collection of circles on the surface whose contacts graph is $G$ such that the circles form a circle packing of the surface that is unique up to a Möbius transformation.
\end{theorem}
We now establish equivalence conditions for circle packings on translation surfaces. These conditions are not the same as those defined by the Koebe-Andreev-Thurston Theorem for spheres:
\begin{remark}
    Unlike in the case of a sphere, equality of contacts graphs is not sufficient for the equality of circle packings up to a Möbius transform on translation surfaces, as 2 circle packings with different tangency patterns are not equivalent on surfaces with genus greater than $0.$ We call 2 circle packings \emph{equivalent} if they have the same contacts graph and the same tangency patterns. 
\end{remark}
\begin{example}
The circle packing in Figure \ref{fig9} has the same contacts graph as $C_3$ on a surface of the same stratum, but has a different tangency pattern. Thus, these 2 packings are not equivalent.
 \begin{figure}[htp]
     \centering
     \includegraphics[width=0.4\textwidth]{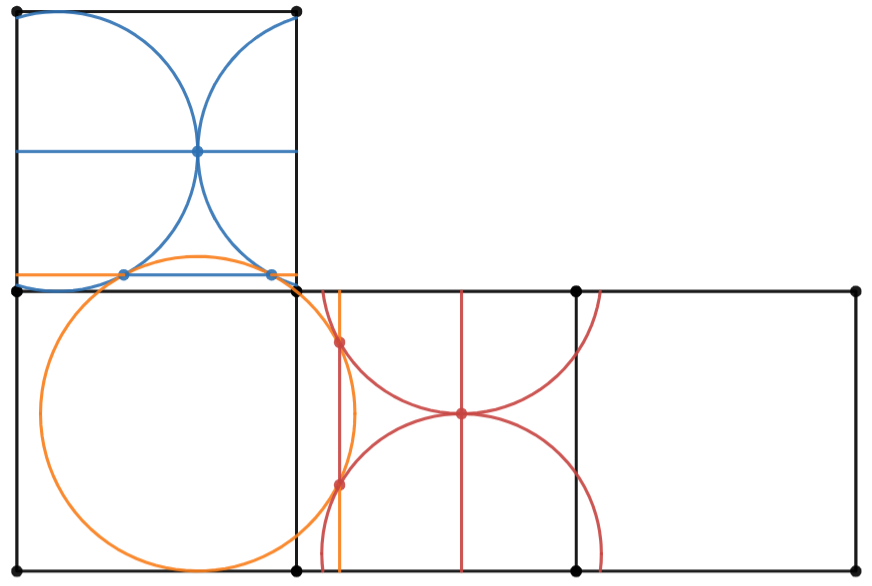}
        \caption{Illustration of non-intersecting tangency segments.}
        \label{fig9}
        \end{figure}
\end{example}

\section{Existence of Equivalent Circle Packings on Distinct Surfaces}
\label{packingexistence}
Consider a circle packing on a translation surface and its contacts graph $G$. One interesting question is whether, for all $G$, there exists a circle packing, unique up to a Möbius transformation, on a different translation surface of the same stratum with the same contacts graph and tangency patterns. In other words, do questions of existence of equivalent circle packings answered by Koebe-Andreev-Thurston Theorem on spheres hold for non-zero genus strata? We answered this question in the negative for specific surfaces in the $\mathcal{H}(2)$ stratum.

\begin{theorem}
    An equivalent packing to $C_3$ cannot be realized on any 4-squared translation surface in $\mathcal{H}(2)$ without applying an affine transformation. 
\end{theorem}
\begin{proof}

We start by noting that having any of the circles contain a singular point cannot yield an equivalent packing to $C_3$ on a 4-squared surface, as shown in Figure \ref{figc3singpoint}. Due to the equivalence outlined in Remark \ref{equivalenceoftranslationsurfacesremark}, we can, without loss of generality, only consider 4-squared surfaces formed by joining 3 squares in a horizontal strip, and then vertically identifying a fourth square to the left-most square in the strip, as in Figure \ref{figc3singpoint}.
   \begin{figure}[htp]
     \centering
     \includegraphics[width=0.4\textwidth]{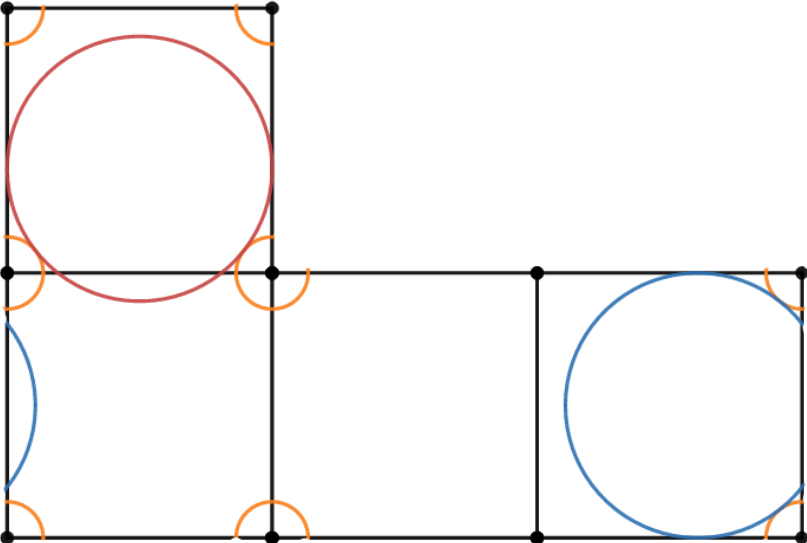}
        \caption{Packing on a 4-squared surface with the same contacts graph as $C_3.$}
        \label{figc3singpoint}
\end{figure}

While the contacts graph in Figure \ref{figc3singpoint} is the same as that of $C_3,$ the 2 packings have differing tangency patterns, since the segment connecting the self-tangency point of the red circle does not intersect the segment connecting the mutual tangency points between the red and orange circles. Since the corresponding segments do intersect in $C_3$, the packing in Figure \ref{figc3singpoint} and $C_3$ are therefore not equivalent. Observe that it was not necessary to center the orange circle at the singular point as was done above.

It remains to observe the effect of not allowing any circles to contain singular points in their interiors. 
\begin{figure}[htp]
    \centering
    \includegraphics[width=.4\textwidth]{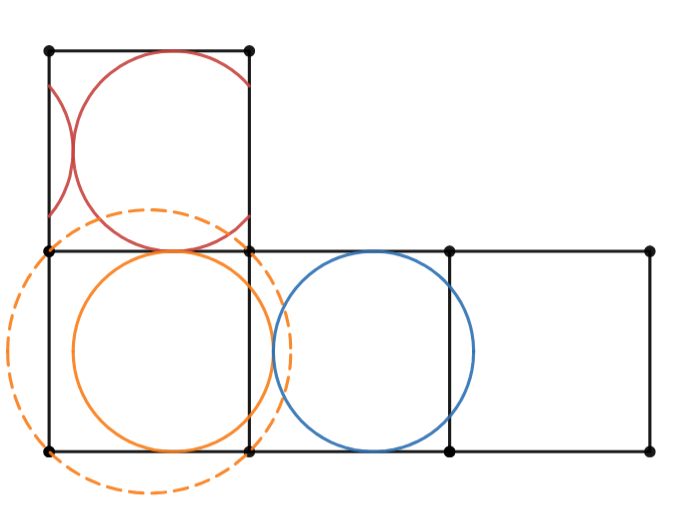}
    \caption{Configuration with no circles containing singular points.}
    \label{figC3IMPOSSIBLEGENERALCASE}
\end{figure}

    In Figure \ref{figC3IMPOSSIBLEGENERALCASE}, the red and blue circles must have radius $\frac{1}{2}$, due to the self-tangency conditions. The orange circle is tangent to both circles, with the tangency patterns preserved for the red circle. In fact, this is the only unique configuration that satisfies the tangency patterns with no circles containing singular points.

However, the blue circle is not tangent to the orange circle a second time, as it should be in order to have the given contacts graph. In fact, the orange circle cannot possibly meet this tangency criterion. The largest possible radius for the orange circle is $\frac{1}{\sqrt{2}}$, depicted as a dashed line in Figure \ref{figC3IMPOSSIBLEGENERALCASE}. Even then, the sums of the diameters of the dashed orange circle and the blue circle are $\sqrt2+1<3,$ so there will only be 1 tangency point even with maximal circle radii.

 So, it is impossible to realize an equivalent packing to $C_3$ on a 4-squared surface of $\mathcal{H}(2)$ if an affine transformation is not applied to this surface.
\end{proof}

\begin{remark}
    In fact, stretching out the 4-squared surface in $\mathcal{H}(2)$ vertically by a factor of $\frac{4}{3}$ allows for an equivalent packing to $C_3$, shown in Figure \ref{fig12}. 

  \begin{figure}[htp]
     \centering
     \includegraphics[width=0.4\textwidth]{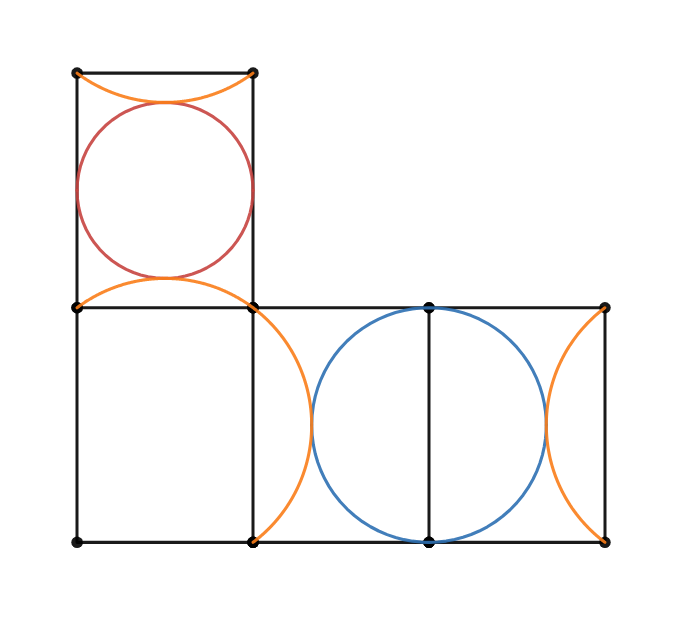}
        \caption{Note that the tangency patterns are the same as in $C_3$.}
        \label{fig12}
        
\end{figure}

\end{remark}
    
\section{Simplest Achievable Contacts Graphs}
\label{simplecontactsgraphs}
Consider the stratum $\mathcal{H}(2)$ and a contacts graph on a translation surface in that stratum. One question of particular interest is how simple this contacts graph can be. We will now answer 2 fundamental questions: first, how many loops can a single vertex of the graph have, and second, how many edges can possibly connect 2 vertices. 

Theorems \ref{H(2)multiedgesthm} and \ref{H(2)multiloopsthm} provide strict bounds on the amount of possible edges between 2 vertices (multi-edges) and loops on a single vertex (multi-loops) in contacts graphs on $\mathcal{H}(2)$ surfaces:
\begin{theorem}
\label{H(2)multiedgesthm}
    A maximum of $8$ multi-edges can be realized on any contacts graph in $\mathcal{H}(2).$
\end{theorem}
\begin{proof}
    The maximal configuration with $8$ multi-edges is achieved on Figure \ref{fig13} below. The vertices of the contacts graph corresponding to the red and blue circles of the circle packing have $8$ edges connecting them. 
  \begin{figure}[htp]
     \centering
     \includegraphics[width=0.4\textwidth]{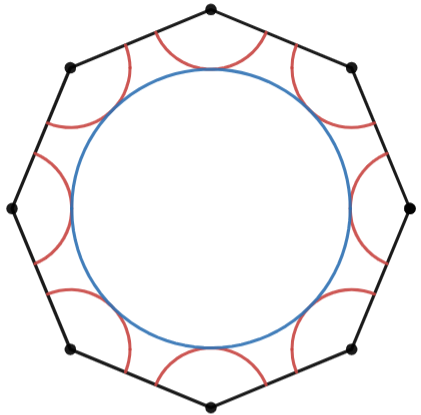}
     \caption{The configuration with maximum multi-edges before gluing of opposite sides.}
     \label{fig13}
\end{figure}

We now demonstrate that no more than $8$ multi-edges can be realized. By Remark \ref{ststooctaremark}, the only polygons on which we must analyze circle packings are octagons with at least 1 axis of rotational symmetry. Having 1 of the circles contain a singular point cannot produce more than $8$ multi-edges, as the second circle can only be tangent to each of the $8$ circular sectors in the complex plane once. If neither circle contains a singular point, there can be a maximum of $5$ multi-edges: 1 in the interior, and 1 along each pair of opposite sides of the octagon. 

Thus, at most $8$ multi-edges can be realized on any contacts graph in $\mathcal{H}(2)$, as desired.

\end{proof}
\begin{theorem}
\label{H(2)multiloopsthm}
     A maximum of $9$ multi-loops can be realized on any contacts graph in $\mathcal{H}(2).$
\end{theorem}
\begin{proof}

Figure \ref{fig88} displays a packing that, when opposite sides are identified, yields the maximum amount of realizable multi-loops in the $\mathcal{H}(2)$ stratum.

\begin{figure}[htp]
     \centering
     \includegraphics[width=0.4\textwidth]{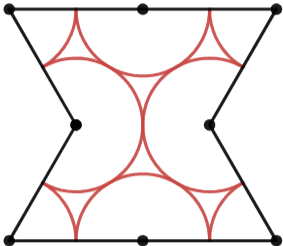}
     \caption{The configuration with maximum multi-loops before gluing of opposite sides.}
     \label{fig88}
\end{figure}

To demonstrate that $10$ multi-loops are not realizable, we examine circular sectors on rotationally symmetric octagons in the plane which are mapped to circle packings on $\mathcal{H}(2)$ surfaces. As in the proof of Theorem \ref{H(2)multiedgesthm}, only packings resulting from gluing of such octagons need to be analyzed by Remark \ref{ststooctaremark}. 

We now perform casework based on the number of vertices of the octagon whose interior angles are greater than $\pi$, which we call \emph{concavities} throughout the proof. There can only be an even number of these, as opposite sides must be parallel. There can also be no more than $4$ concavities, as the sum of the interior angles of an octagon with $6$ concavities is greater than $6\pi.$
\begin{enumerate}
    \item If there are no concavities, there are only 4 possible multi-loops corresponding to $4$ pairs of opposite sides of the octagon.
    \item If there are 2 concavities, the vertices with angles greater than $\pi$ must be opposite each other. A circle that does not contain the singular point can only have at most $4$ multi-loops by the argument in the previous case. If the circle contains a singular point, there can be at most $1$ self-tangency in the interior of the octagon and $4$ self-tangencies on the edges, which again correspond to $4$ pairs of opposite sides, for a maximum of $5$ multi-loops in this case.
    \item Finally, if there are 4 concavities, we again maximize the amount of multi-loops with a circle containing the singular point. There can be a maximum of 4 self-tangencies on the edges of the polygon and 5 in the interior, since it is well-known that 4 circles with the same radius in the plane can have a maximum of 5 mutual tangency points. In fact, the only such configuration (unique up to an affine transformation in the plane) is pictured in Figure \ref{9touchings}, and placing alternating vertices of the octagon at the centers of these circles yields the configuration shown in Figure \ref{fig88}. This case therefore yields a maximum of $9$ multi-loops. 
    \begin{figure}[htp]
     \centering
     \includegraphics[width=0.25\textwidth]{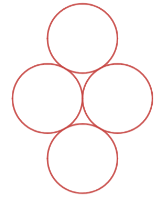}
     \caption{Configuration of 4 circles in the plane with 5 mutual tangencies.}
     \label{9touchings}
\end{figure}

\end{enumerate} 

To conclude, 9 is the maximum possible number of multi-loops of any circle packing on any surface in $\mathcal{H}(2)$.
\end{proof}
Theorems \ref{4gmultiloopsthm} and \ref{4gmultiedgesthm} demonstrate that certain numbers of multi-edges and multi-loops can be realized on contacts graphs of packings on general genus $g$ translation surfaces with 1 singular point.
\begin{theorem}
\label{4gmultiloopsthm}
It is possible to realize $4g$ multi-loops on a contacts graph on at least 1 surface of the genus $g$ stratum $\mathcal{H}(2g-2).$
\end{theorem}

\begin{proof}
    Consider the regular polygon $\mathcal{Q}_0=A_2A_4\cdots A_{4g}$ with $2g$ vertices in the complex plane. For each side $\overline{A_{2i}A_{2i+2}}$ of $\mathcal{Q}_0$, construct the equilateral triangle $A_{2i}A_{2i+1}A_{2i+2}$ so that $A_{2i+1}$ is not in the interior of $\mathcal{Q}_0.$ Let $\mathcal{Q}=A_1A_2\cdots A_{4g-1}A_{4g}.$ 
    
    For each $2i+1,$ center a circular sector $C_{2i+1}$ contained in $\mathcal{Q}$ with angle measure $\frac{\pi}{3}$ and radius $\frac{A_1A_2}{2}$ at $A_{2i+1},$ and for each $2i,$ center a circular sector $C_{2i}$ contained in $\mathcal{Q}$ with angle measure $\left(\frac{2g-2}{2g}+\frac{2}{3} \right) \pi$ and radius $\frac{A_1A_2}{2}$ at $A_{2i}.$ Observe that members of the set $\{C_n | 1\le n \le 4g \}$ have $2$ mutual tangency points for each of the $2g$ pairs of opposite sides of $\mathcal{Q}$, as well as $2g$ mutual tangency points at the midpoints of the $2g$ sides of $\mathcal{Q}_0.$
    
    Now, we identify the opposite sides of $\mathcal{Q}$ to form a surface $X_{\mathcal{Q}}$ in the $\mathcal{H}(2g-2)$ stratum: $X_{\mathcal{Q}}$ will have exactly 1 singular point of cone angle $(4g-2)\pi$ and, consequently, of order $2g-2$. Note that all sectors $C_n$ get mapped to a single circle $C$ centered at the singular point of $X_{\mathcal{Q}}$. The number of self-tangencies of $C$ is equal to the number of unique tangency points between members of $\{C_n | 1\le n \le 4g \}.$ Note that each of the $2g$ pairs of tangency points in the complex plane that are at midpoints of pairs of opposite sides of $\mathcal{Q}$ get identified to 1 point on the surface, contributing $2g$ to the number of self-tangencies of $C$. The tangency points at midpoints of the $2g$ sides of $\mathcal{Q}_0$ are in the interior of $\mathcal{Q}$, and so contribute another $2g$ to $C$'s self-tangencies. Thus, it is possible for a circle packing realized on a surface in $\mathcal{H}(2g-2)$ to have $4g$ self-tangencies, as desired.
\end{proof}

\begin{theorem}
\label{4gmultiedgesthm}
It is possible to realize $4g$ multi-edges on a contacts graph on at least 1 surface of the genus $g$ stratum $\mathcal{H}(2g-2).$
\end{theorem}
\begin{proof}
    Consider the regular $4g$-gon $\mathcal{Q}_{\text{reg}}=A_1A_2\cdots A_{4g}.$ Identification of opposite sides of $\mathcal{Q}_{\text{reg}}$ yields a genus $g$ surface in $\mathcal{H}(2g-2).$ At each vertex $A_i$ of $\mathcal{Q}_{\text{reg}},$ center a circular sector of angle measure $\frac{2g-1}{2g}\pi$ and radius $r<\frac{A_1A_2}{2}.$ Next, center a circle $C_1$ with radius $\frac{A_1A_2}{2\sin \left(\frac{\pi}{4g} \right)}-r$ at the center of $\mathcal{Q}_{\text{reg}}.$
    
    After identification of opposite side lengths, each circle centered at a vertex of $\mathcal{Q}_{\text{reg}}$ gets mapped to a single circle $C_2,$ and there will be $4g$ tangency points between $C_1$ and $C_2$ corresponding to $4g$ vertices of $\mathcal{Q}_{\text{reg}}.$ Thus, we have constructed a packing whose contacts graph has $4g$ multi-edges on a surface in $\mathcal{H}(2g-2).$
\end{proof}
Similarly, on genus $g$ translation surfaces with 2 singular points of equal order, the possibility of certain packing complexities is demonstrated in Theorems \ref{2g+1multiloopsthm} and \ref{thmmultiedgesg}.
\begin{theorem}
\label{2g+1multiloopsthm}
   It is possible to realize $2g+1$ multi-loops on a contacts graph on at least 1 surface of the genus $g$ stratum $\mathcal{H}(g-1,g-1).$
\end{theorem}

\begin{proof}
        First, note that the surface formed by identifying opposite edges of a regular polygon $\mathcal{Q}_{\text{reg}}=A_1A_2\cdots A_{4g+2}$ is in $\mathcal{H}(g-1,g-1).$ Now, center $2g+1$ circular sectors with radii $A_1A_2 \cdot \sin\left(\frac{g}{2g+1}\pi\right)$ at each even-numbered vertex $A_{2i}.$ This will yield $2g+1$ tangency points in the interior $\mathcal{Q}_{\text{reg}}$, which correspond to $2g+1$ multi-loops on the contacts graph of the $\mathcal{H}(g-1,g-1)$ surface, as desired.
\end{proof}

\begin{theorem}
\label{thmmultiedgesg}
    It is possible to realize $2g+2$ multi-edges on a contacts graph on at least 1 surface of the genus $g$ stratum $\mathcal{H}(g-1,g-1).$
\end{theorem}

\begin{proof}
    Consider 2 regular $(2g+2)$-sided polygons, $\mathcal{Q}_1$ and $\mathcal{Q}_2$. Choose a side of each $(2g+2)$-gon, glue $\mathcal{Q}_1$ and $\mathcal{Q}_2$ together along those 2 sides, and then remove the edge along which they were glued. The result is a $(4g+2)$-gon $\mathcal{Q}$, which, as discussed in the proof of Theorem \ref{2g+1multiloopsthm}, becomes a surface in $\mathcal{H}(g-1,g-1)$ upon identification of opposite edges. 
    
    Now, inscribe a circle $C_1$ in $\mathcal{Q}_1,$ and likewise inscribe a circle $C_2$ in $\mathcal{Q}_2$. Identifying opposite edges of $\mathcal{Q}$ yields $2g+2$ mutual tangencies between $C_1$ and $C_2$, corresponding to the $2g+2$ edges of $\mathcal{Q}_1.$ Thus, the contacts graph of the packing consisting of $C_1$ and $C_2$ on the surface formed by identifying opposite sides of $\mathcal{Q}$ has $2g+2$ multi-edges, satisfying the desired condition.
\end{proof}

\begin{corollary}
\label{H(1,1)multiloopsandedges}
    Up to $5$ multi-loops and $6$ multi-edges can be realized on a contacts graph on at least 1 surface in $\mathcal{H}(1,1).$
\end{corollary}

\begin{proof}
    Substitute $g=2$ in Theorems \ref{2g+1multiloopsthm} and \ref{thmmultiedgesg}.
\end{proof}

\section{Research Directions}
\label{researchdirections}

Theorems \ref{H(2)multiedgesthm} and \ref{H(2)multiloopsthm} provide a bound on the complexity of contacts graphs in $\mathcal{H}(2)$, while no such bound is presented for contacts graphs on $\mathcal{H}(1,1)$ surfaces in Corollary \ref{H(1,1)multiloopsandedges}. A natural research direction is therefore establishing such a bound in order to complete our results on complexity in the genus $2$ stratum. Progress can also be made in bounding the numbers of possible multi-loops and multi-edges on surfaces in the genus $g$ strata $\mathcal{H}(2g-2)$ and $\mathcal{H}(g-1,g-1),$ continuing the progress made in Theorems \ref{4gmultiloopsthm}, \ref{4gmultiedgesthm}, \ref{2g+1multiloopsthm}, and \ref{thmmultiedgesg}. 

Additionally, only partial progress was made towards describing equivalent packings that can be realized on distinct, same-stratum surfaces. It would be particularly interesting to investigate the existence of packings that can be realized on any surface in certain strata, and to give the most complex such packings in $\mathcal{H}(2),$ $\mathcal{H}(1,1),$ and potentially higher-genus strata.

\section{Acknowledgements}
I would like to sincerely thank my PRIMES-USA mentor, Professor Sergiy Merenkov, for his guidance during the reading and research phases of the program. Without him, this work would not have been possible. I am grateful to MIT PRIMES-USA for giving me the invaluable opportunity to conduct this research. I am also thankful to Doctor Tanya Khovanova for her comments on my initial reports and this paper.

\printbibliography
\end{document}